\documentclass[review,onefignum,onetabnum]{siamart171218}
\usepackage{multirow}
\usepackage{amsmath}
\usepackage{color}
\usepackage{amssymb}
\usepackage{mathrsfs}
\usepackage{ulem} 

\newcommand\s{\mathbf{s}}
\newcommand\g{\mathbf{g}}
\newcommand\z{\mathbf{z}}

\newcommand\vvec{\mathbf{v}}

\newcommand\x{\mathbf{x}}
\newcommand\y{\mathbf{y}}

\newcommand\dt{\Delta t}
\newcommand\I{\mathbf{I}}

\newtheorem{assumption}{Assumption}[section]

\usepackage{lipsum}
\usepackage{amsfonts}
\usepackage{graphicx}
\usepackage{epstopdf}
\usepackage{algorithmic}
\newsiamremark{remark}{Remark}
\newsiamremark{hypothesis}{Hypothesis}
\crefname{hypothesis}{Hypothesis}{Hypotheses}
\newsiamthm{claim}{Claim}

\definecolor{midbluepurple}{RGB}{134,0,102}

\newcommand{\TheTitle}{A Second-Order
Method for Locating Critical Points}
\newcommand{\TheAuthors}{Q. Du, B. Shi}

\headers{\TheTitle}{\TheAuthors}
\title{{\TheTitle}\thanks{The work was supported by NSF DMS-2309245.
}}

\author{
Qiang Du\thanks{Department of Applied Physics and Applied Mathematics, and Data Science Institute, Columbia University, New York, NY 10027, USA (\email{qd2125@columbia.edu}).}
\and Baoming Shi\thanks{Department of Applied Physics and Applied Mathematics, Columbia University, New York, NY 10027, USA (\email{bs3705@columbia.edu}).}
}
\ifpdf
\hypersetup{
  pdftitle={\TheTitle},
  pdfauthor={\TheAuthors}
}
\fi
\renewcommand{\TheTitle}{A Second-Order
Method for Locating Critical Points of Prescribed Index on Energy Landscapes}

\begin{document}
\maketitle

\begin{abstract} 
Exploring high dimensional energy landscape is a challenging problem in many applications where both ground states and transition states offer important information about the underlying physical systems. To facilitate the computational exploration, we propose a cubic-regularized second-order method for locating critical points of prescribed index, including both energy minima (with an index $0$) and transition states (saddle points with positive index). At each step of the iteration, the proposed method involves several ingredients, such as updating the eigen-directions of the Hessian according to the prescribed index, reflecting the gradient and the Hessian along selected unstable eigen-directions, if any, and constructing the next step by solving a cubic-regularized subproblem. Under suitable assumptions, we show that, when the eigenspace and the cubic subproblem are solved exactly, the iteration converges locally and quadratically to a critical point of the prescribed index. We further establish convergence results in the inexact setting, where both the eigenspace computation and the solution of the cubic subproblem are performed approximately. We show that the resulting iteration retains local linear convergence, with the convergence rate depending on the accuracy of the inexact computations. Moreover, we prove that the fixed points of the algorithm are precisely critical points of the prescribed index. In addition, we introduce an adaptive strategy for updating the cubic regularization parameter to improve robustness and computational efficiency of the proposed algorithm. Numerical experiments confirm the predicted convergence behavior, demonstrate the effectiveness of the adaptive strategy, and illustrate the capability of the proposed method for locating critical points and constructing solution landscapes.
\end{abstract}



\section{Introduction}
Exploring the energy landscape is a fundamental problem in many high-dimensional systems arising in a wide range of scientific applications \cite{onuchic1997theory,wales2006energy,wales2018exploring,wolynes1995navigating}. Such an exploration involves not only identifying local minima, which represent stable or metastable configurations and include the ground states of the system, but also characterizing the broader collection of critical points that structure the landscape. 
In particular, saddle points play a central role in mediating transitions between metastable states and in determining the connectivity of their basins of attraction. The locations, energies, and Morse indices of these saddle points, where the Morse index is defined as the number of negative eigenvalues of the Hessian, provide important information about transition barriers, competing pathways, and the global organization of the energy landscape \cite{vanden2010transition}. Consequently, a comprehensive understanding of high-dimensional energy landscapes requires the systematic computation and characterization of critical points across different indices in identifying transition mechanisms that cannot be inferred from local minima alone.

Locating saddle points (or critical points of positive index) is substantially more difficult than identifying local minima, since saddle points are inherently unstable and their unstable eigendirections are typically unknown a priori. A variety of numerical approaches have therefore been developed for saddle-point computation. Broadly speaking, these methods can be categorized into path-finding approaches \cite{weinan2002string,vanden2010transition} and surface-walking approaches \cite{zhang2016recent}. The former seeks minimum-energy paths connecting metastable states, along which local energy maxima typically correspond to transition states or index-one saddle points. We are primarily concerned with the latter, which generate trajectories or iterative updates directly on the energy landscape and drive the system toward saddle points. Representative examples include the min-max method \cite{li2001minimax}, the iterative minimization formulation \cite{gao2015iterative}, gentlest ascent dynamics \cite{gad,quapp2014locating,liu2023constrained}, dimer dynamics and high-index saddle dynamics (HiSD) \cite{henkelman1999dimer, zhangdu2012,2019High,su2025improved}.

A notable advantage of methods like HiSD is their ability to target critical points of a prescribed index, where an index-$k$ critical point is a critical point whose Hessian has exactly $k$ negative eigenvalues. For $k=0$, one recovers methods for computing local minima. This makes HiSD particularly suitable for constructing solution landscapes that contain multiple critical points and their connections \cite{yin2020construction}. However, HiSD with constant step size generally exhibits only a linear local convergence rate \cite{luo2022sinum}. Since the construction of a solution landscape typically requires repeated searches of local minima and saddle points, the resulting computational cost can become substantial. Newton's method, on the other hand, achieves quadratic local convergence to nondegenerate critical points, but offers no mechanism for controlling the index of the limiting critical point \cite{nocedal1999numerical}. Thus, the two approaches exhibit complementary strengths: HiSD provides index selectivity but relatively slow local convergence, whereas Newton's method provides fast convergence without index selectivity. For applications such as solution-landscape construction and transition pathway computation, it is desirable to combine these two properties. This motivates us to develop a method that targets critical points of a prescribed index while achieving faster convergence.

In this paper, we develop a cubic-regularized second-order framework for computing critical points of a prescribed index. The key algorithmic advance is rooted in the combination of the index-selection mechanism induced by the eigenspace associated with the $k$ smallest eigenvalues of the Hessian with a second-order cubic model \cite{nesterov2006cubic}. By reflecting the first- and second-order information along the selected unstable eigen-directions for critical points with positive index, the resulting local model transforms the desired saddle structure into a form suitable for cubic regularization. This construction provides a natural bridge between prescribed-index critical-point search and modern second-order optimization techniques. We establish a local convergence theory showing that the proposed iteration converges quadratically to a nondegenerate index-$k$ critical point when the eigenspace and cubic subproblems are solved exactly. We further quantify the effect of inexact subproblem solutions on the local convergence rate and establish a linear convergence rate whose contraction factor depends on the accuracy of the subproblem solutions. We also characterize the fixed points of the algorithm and show that they correspond precisely to critical points of the prescribed index. In practice, we exploit the standard structure of the cubic subproblem to incorporate existing efficient solvers and develop an adaptive regularization strategy for improved robustness and efficiency. Numerical experiments validate the theoretical results and demonstrate the effectiveness of the method in locating prescribed-index critical points and constructing solution landscapes.

The paper is organized as follows. In \Cref{sec: Saddle point}, we present the proposed algorithm, together with its theoretical properties under exact subproblem solutions. In \Cref{sec:inexact}, we extend the analysis to the inexact setting, where both the unstable eigenspace and the cubic subproblem are computed approximately. In \Cref{sec: methods for subproblem}, we discuss numerical methods for solving the cubic subproblem and introduce an adaptive strategy for updating the cubic regularization parameter. In \Cref{sec: solution landscape}, we describe how the proposed algorithm can be used to construct solution landscapes. In \Cref{sec: Numerics}, we present numerical experiments. Finally, conclusions and discussions are presented in \Cref{sec: conclusion}.

\section{Critical points and the second-order method}\label{sec: Saddle point}
Consider a system characterized by an energy function $f:\mathcal{H}\rightarrow \mathbb{R}$, where $\mathcal{H}$ is a real Hilbert space and $f$ is assumed to be sufficiently smooth. A critical point $\x^*\in\mathcal{H}$ is a configuration at which the Fr\'echet gradient vanishes, i.e., $\nabla f(\x^*)=0$. Its local stability is characterized by the Hessian operator $\nabla^2 f(\x^*)$. If the negative spectral subspace of the self-adjoint Hessian $\nabla^2 f(\x^*)$ has dimension $k$, then $\x^*$ is referred to as an index-$k$ critical point, and the corresponding negative spectral subspace is called the unstable subspace. For simplicity, throughout this work we restrict our analysis to the finite-dimensional setting $\mathcal{H}=\mathbb{R}^d$, $d\geq 2$, and assume that $f\in C^3(\mathbb{R}^d)$ has finitely many non-degenerate critical points. In this case, if $k=0$ (or $k=d$), i.e., $\nabla^2 f(\x^*)$ is positive definite (or negative definite), then $\x^*$ is a local minimum (or maximum) \cite{nocedal1999numerical}. If $1\leq k\leq d-1$, $\x^*$ is an index-$k$ saddle point. We denote the corresponding orthonormal eigenvectors by $\{\vvec_1^*,\cdots,\vvec_k^*\}$, which are referred to as the unstable eigenvectors or unstable directions.

There are various numerical methods for locating critical points of prescribed indices. For an index-$k$ saddle point  $\x^*$ ($k\geq 1$), a representative example is the saddle dynamics \cite{2019High, gad}, defined by
\begin{equation}
\dot{\x}
=
-\left(\I-2\sum_{i=1}^k \vvec_i\vvec_i^\top\right)\nabla f(\x),
\label{eq: SD}
\end{equation}
where $\{\vvec_i\}_{i=1}^k$, with $1\leq k\leq d-1$, are orthonormal eigenvectors associated with the $k$ smallest eigenvalues of $\nabla^2 f(\x)$, and $\I$ denotes the identity operator. For $k=0$ or $k=d$, the dynamics recovers the conventional gradient flow (descent or ascent) for energy minimization or maximization respectively.
Under suitable conditions, the stable equilibria of \eqref{eq: SD} correspond to index-$k$ saddle points \cite{2019High, yin2021searching}. A straightforward Euler discretization of \eqref{eq: SD} gives
\begin{equation}
\label{eq: first-order method}
\x_{j+1}
=
\x_j
-
\dt
\left(
\I-2\sum_{i=1}^k
\vvec_{j,i}\vvec_{j,i}^\top
\right)
\nabla f(\x_j).
\end{equation}
As a first-order iteration, however, \eqref{eq: first-order method} generally converges only linearly, with its local convergence rate depending on the conditioning of the Hessian \cite{luo2022sinum, du2026convergence}. In contrast, Newton's method enjoys quadratic local convergence to any nondegenerate critical point. It does not, however, distinguish the Morse index of the limiting point: depending on the initial condition, it may converge to a local minimum, a local maximum, or a saddle point of an undesired index. Hence, Newton's method provides fast local convergence but lacks the index-selection mechanism required for prescribed-index critical point computation in energy landscape exploration. 

This observation naturally leads us to second-order optimization methods. In particular, cubic regularization was originally developed as a globalization of Newton's method for minimization \cite{nesterov2006cubic}. At each iteration, it replaces the quadratic Newton model by a cubic-regularized model and seeks its minimizer. From the viewpoint of Morse index, such a method is designed to locate index-$0$ critical points, namely local minima. Our goal is to extend this cubic-regularization framework from index-$0$ local minima to critical points of an arbitrary prescribed index $k$. To this end, we incorporate the index-selection mechanism into the cubic model by reflecting the gradient and Hessian along the eigenspace associated with the $k$ smallest eigenvalues. This leads to the following cubic-regularization method for index-$k$ critical point in \Cref{algorithm 1.1}, where $\{\vvec_i\}_{i=1}^k=\mathrm{Eigensolver}(\nabla^2 f(\x_j))$ denotes the computation of the eigenvectors associated with the $k$ smallest eigenvalues of $\nabla^2 f(\x_j)$.

\begin{algorithm}
\caption{Cubic regularization method for index-$k$ critical point}
\label{algorithm 1.1}
\begin{algorithmic}
\STATE{Given $\x_0\in\mathbb{R}^d$ and maximum number of iterations $n_{max}$}

\FOR{$j$ in $\{0,\cdots,n_{max}-1\}$}
\STATE{Update the  invariant subspace associated with the $k$ smallest eigenvalues: $$\{\vvec_i\}_{i=1}^k=\mathrm{Eigensolver}(\nabla^2 f(\x_j))$$}
\STATE{Construct the reflected gradient and reflected Hessian: $$\mathbf{r}_j=\left(\I-\sum_{i=1}^k2\vvec_i\vvec_i^\top\right)\nabla f(\x_j), \ B_j=\mathrm{sym}\left(\left(\I-\sum_{i=1}^k2\vvec_i\vvec_i^\top\right) \nabla^2f(\x_j)\right)$$}
\STATE{Solve the subproblem: $$\mathbf{s}_j\in\operatorname*{arg\,min}_{\mathbf{s}\in\mathbb{R}^d}  f(\x_j)+\mathbf{r}_j^\top \mathbf{s}+\frac{1}{2}\mathbf{s}^\top B_j\mathbf{s}+\frac{\sigma_j}{3}\|\mathbf{s}\|_2^3$$}
\STATE{Update $\x_{j+1}= \x_j+\mathbf{s}_j$}
\ENDFOR
\end{algorithmic}
\end{algorithm}

To analyze the convergence of the proposed method, we first recall the following standard global optimality characterization for cubic subproblems \cite{nesterov2006cubic,cartis2011adaptive}.
\begin{lemma}
\label{prop:cubic-opt}
Let $\x ,\mathbf{r}\in\mathbb{R}^d$, let $B=B^\top \in \mathbb{R}^{d\times d}$, and let $\sigma>0$. A vector $\mathbf{s}$ is a
global minimizer of
\[
   q(\mathbf{s})=f(\x)+\mathbf{r}^\top \mathbf{s}+\frac12 \mathbf{s}^\top B\mathbf{s}+\frac{\sigma}{3}\|\mathbf{s}\|^3_2
\]
if and only if
\begin{equation}
   (B+\lambda \I)\mathbf{s}=-\mathbf{r},\quad
   B+\lambda \I\succeq0,\quad
   \lambda=\sigma\|\mathbf{s}\|_2.
   \label{eq:cubic-opt}
\end{equation}
In particular, if $B\succ0$, then $q(\mathbf{s})$ is strongly convex and its global minimizer is unique.
\end{lemma}

\begin{assumption}[Local regularity and strict critical point]
\label{assumption:local}
Let $\x^*$ be a strict index-$k$ critical point. There is a neighborhood
$U=\{\x,\|\x-\x^*\|_2\leq r_0\}$ of $\x^*$ and constants $M>0$, $L\geq \mu>0$ such that $\|\nabla^2f(\x)-\nabla^2f(\y)\|_2\leq M\|\x-\y\|_2, \x,\y\in U$ and the spectrum of $\nabla^2f(\x), \x \in U$ satisfies that
$$
-L \leq \lambda_1\leq \cdots \leq \lambda_k\leq -\mu <0<\mu \leq \lambda_{k+1}\leq \cdots \leq \lambda_d\leq L,
$$
with the corresponding inequalities on the negative or positive eigenvalues omitted when $k=0$ or $k=d$.
\end{assumption}

\begin{lemma}[Local positivity of the reflected Hessian]
\label{lem:positive}
Under Assumption~\ref{assumption:local}, the reflected Hessian $B(\x)=\left(\I-\sum_{i=1}^k2\vvec_i(\x)\vvec_i(\x)^\top\right) \nabla^2f(\x)\succeq \mu \I$ is a positive-definite matrix. Consequently, the reflected cubic subproblem has a unique global minimizer for
every $\x\in  U$.
\end{lemma}

\begin{proof}
Since $\{\vvec_i(\x)\}_{i=1}^k$ are the eigenvectors associated with the $k$ smallest eigenvalues of $\nabla^2 f(\x)$, the eigenvalues of $B(\x)$ are $-\lambda_1,\ldots,-\lambda_k,\lambda_{k+1},\ldots,\lambda_d$. By \Cref{assumption:local}, all these eigenvalues are bounded below by $\mu$. Hence, $B(\x)\succeq \mu \I$.
\end{proof}

\begin{proposition}[Local quadratic convergence]
\label{thm:quadratic}
Suppose Assumption~\ref{assumption:local} holds and $\x_j$ is generated by \Cref{algorithm 1.1} with
$0<\sigma_{min}\leq\sigma_j\leq \sigma_{max}$, $\forall j\in\mathbb{N}$. Define
\begin{equation}
    C=\frac{2L^2\left(\sigma_{max}+\frac{M}{2}\right)}{\mu^3},
    \label{eq:quadratic-constant}
\end{equation}
and choose $r>0$ such that
\begin{equation}
    r\leq
    \min\left\{
    \frac{\mu r_0}{\mu+L},
    \frac{\mu^2}{M(\mu+L)},\frac{1}{C}
    \right\}.
    \label{eq:quadratic-radius}
\end{equation}
If $\|\x_0-\x^*\|_2\leq r$, then $\x_j\in U$ for all $j\in\mathbb{N}$ and
\begin{equation}
   \|\x_{j+1}-\x^*\|_2
   \leq C\|\x_j-\x^*\|_2^2,
   \quad \forall j \in \mathbb{N}.
   \label{eq:quadratic-rate}
\end{equation}
\end{proposition}

\begin{proof}
In the neighborhood from Lemma~\ref{lem:positive}, the subproblem is
strongly convex and its unique minimizer satisfies
\begin{equation}
   \left(\I-\sum_{i=1}^k2\vvec_i\vvec_i^\top\right)\nabla f(\x_j)
   +\left(\I-\sum_{i=1}^k2\vvec_i\vvec_i^\top\right)
   \nabla^2f(\x_j)\s_j
   +\sigma_j\|\s_j\|_2\s_j=0,
   \label{eq:model-stationarity}
\end{equation}
by \Cref{prop:cubic-opt}. Taking the inner product with $\s_j$, using
\[
\left(\I-\sum_{i=1}^k2\vvec_i\vvec_i^\top\right)
\nabla^2f(\x_j)\succeq\mu\I,
\quad
\left\|
\left(\I-\sum_{i=1}^k2\vvec_i\vvec_i^\top\right)
\nabla f(\x_j)
\right\|_2
=
\|\nabla f(\x_j)\|_2,
\]
gives
\[
\mu\|\s_j\|_2^2+\sigma_j\|\s_j\|_2^3
   \leq
   \|\nabla f(\x_j)\|_2\|\s_j\|_2,
\]
and therefore
\begin{equation}
   \|\s_j\|_2
   \leq
   \frac{\|\nabla f(\x_j)\|_2}{\mu}.
   \label{eq:step-bound}
\end{equation}

For every $\x\in U$, since $\nabla f(\x^*)=0$ and the line segment
between $\x^*$ and $\x$ is contained in $U$, Assumption~\ref{assumption:local}
gives
\begin{equation}
   \|\nabla f(\x)\|_2
   =
   \left\|
   \int_0^1
   \nabla^2f\big(\x^*+t(\x-\x^*)\big)(\x-\x^*)\,dt
   \right\|_2  \leq
   L\|\x-\x^*\|_2.
\label{eq:gradient-upper-bound}
\end{equation}
Hence, whenever $\|\x_j-\x^*\|_2\leq r$,
\begin{equation}
   \|\s_j\|_2
   \leq
   \frac{L}{\mu}\|\x_j-\x^*\|_2
   \leq
   \frac{L}{\mu}r.
   \label{eq:step-error-bound}
\end{equation}
Consequently, for every $t\in[0,1]$,
\begin{equation}
   \|\x_j+t\s_j-\x^*\|_2
   \leq
   \|\x_j-\x^*\|_2+t\|\s_j\|_2  \leq
   \left(1+\frac{L}{\mu}\right)r.
\label{eq:segment-bound}
\end{equation}
By \eqref{eq:quadratic-radius}, $\left(1+\frac{L}{\mu}\right)r
   \leq
   \min\left\{r_0,\frac{\mu}{M}\right\}$, thus, the whole line segment
$\{\x_j+t\s_j:t\in[0,1]\}$ is contained in $U$, and
\begin{equation}
   \|\x_j+\s_j-\x^*\|_2\leq\frac{\mu}{M}.
   \label{eq:new-point-radius}
\end{equation}
Multiplying \eqref{eq:model-stationarity} by
$\I-\sum_{i=1}^k2\vvec_i\vvec_i^\top$ and using $\left(\I-\sum_{i=1}^k2\vvec_i\vvec_i^\top\right)^2=\I$ yields
\begin{equation}
   \nabla f(\x_j)+\nabla^2 f(\x_j)\s_j
   =
   -\sigma_j\|\s_j\|_2
   \left(\I-\sum_{i=1}^k2\vvec_i\vvec_i^\top\right)\s_j.
   \label{eq:inexact-newton}
\end{equation}
The Hessian Lipschitz condition gives
\begin{equation}
   \|\nabla f(\x_j+\s_j)-\nabla f(\x_j)
   -\nabla^2f(\x_j)\s_j\|_2
   \leq
   \frac{M}{2}\|\s_j\|_2^2.
   \label{eq:taylor-gradient-remainder}
\end{equation}
Combining \eqref{eq:taylor-gradient-remainder} with
\eqref{eq:inexact-newton} and using the orthogonality of
$\I-\sum_{i=1}^k2\vvec_i\vvec_i^\top$, we obtain
\begin{equation}
   \|\nabla f(\x_j+\s_j)\|_2
   \leq
   \left(\sigma_j+\frac{M}{2}\right)\|\s_j\|_2^2
   \leq
   \left(\sigma_{max}+\frac{M}{2}\right)\|\s_j\|_2^2.
   \label{eq:new-gradient-bound}
\end{equation}

We next establish an explicit local error bound. For any $\y\in U$,
the Hessian Lipschitz condition and $\nabla f(\x^*)=0$ give
$%
   \nabla f(\y)
   =
   \nabla^2f(\x^*)(\y-\x^*)+\mathbf{R}(\y)$,
where
\[
   \|\mathbf{R}(\y)\|_2
   \leq
   \frac{M}{2}\|\y-\x^*\|_2^2.
\]
Since every eigenvalue of $\nabla^2f(\x^*)$ has absolute value at least
$\mu$,
\[
   \|\nabla^2f(\x^*)(\y-\x^*)\|_2
   \geq
   \mu\|\y-\x^*\|_2.
\]
Therefore, if
$\|\y-\x^*\|_2\leq\mu/M$, then
$$
   \|\nabla f(\y)\|_2
   \geq
   \mu\|\y-\x^*\|_2
   -\frac{M}{2}\|\y-\x^*\|_2^2  
   \geq
   \frac{\mu}{2}\|\y-\x^*\|_2,
$$
and hence
\begin{equation}
   \|\y-\x^*\|_2
   \leq
   \frac{2}{\mu}\|\nabla f(\y)\|_2.
   \label{eq:explicit-error-bound}
\end{equation}

Applying \eqref{eq:explicit-error-bound} to
$\y=\x_j+\s_j$, which is justified by
\eqref{eq:new-point-radius}, and using
\eqref{eq:new-gradient-bound},
\eqref{eq:step-bound}, and
\eqref{eq:gradient-upper-bound}, gives
\begin{equation}
\begin{aligned}
   \|\x_{j+1}-\x^*\|_2&=
   \|\x_j+\s_j-\x^*\|_2 \leq
   \frac{2}{\mu}
   \|\nabla f(\x_j+\s_j)\|_2  \\
   &\leq
   \frac{2}{\mu}
   \left(\sigma_{max}+\frac{M}{2}\right)
   \|\s_j\|_2^2 \leq
   \frac{2}{\mu}
   \left(\sigma_{max}+\frac{M}{2}\right)
   \frac{\|\nabla f(\x_j)\|_2^2}{\mu^2}\\
   &\leq
   \frac{2L^2\left(\sigma_{max}+\frac{M}{2}\right)}{\mu^3}
   \|\x_j-\x^*\|_2^2 =C\|\x_j-\x^*\|_2^2.
\end{aligned}
\label{eq:quadratic-estimate}
\end{equation}
Moreover, by \eqref{eq:quadratic-radius}, $Cr\leq1$. Hence, if $\|\x_j-\x^*\|_2\leq r$, then
\[
   \|\x_{j+1}-\x^*\|_2
   \leq
   C\|\x_j-\x^*\|_2^2
   \leq
   Cr\|\x_j-\x^*\|_2
   \leq
   \|\x_j-\x^*\|_2
   \leq r.
\]
Since $\|\x_0-\x^*\|_2\leq r$, induction gives
$\|\x_j-\x^*\|_2\leq r$ for every $j\in\mathbb{N}$, and
\eqref{eq:quadratic-rate} follows.
\end{proof}

The following proposition shows that the cubic regularization method in \Cref{algorithm 1.1} is tailored specifically to finding index-$k$ critical points.
\begin{proposition}
\label{thm:stability}
Fix $k\in\{0,\ldots,d\}$ in \Cref{algorithm 1.1}. Suppose all the critical points of $f$ are non-degenerate; then, $\bar \x$ is a fixed point of \Cref{algorithm 1.1} if and only if $\bar \x$ is an index-$k$ critical point of $f$, where $\bar \x$ is called a fixed point if the zero step is a global minimizer of the cubic subproblem.
\end{proposition}

\begin{proof}
At an index-$k$ critical point $\bar \x$, $\nabla f(\bar \x)=\mathbf{0}$. The linear coefficient of the subproblem therefore vanishes. Let
\[
B=\left(\I-\sum_{i=1}^k2\vvec_i\vvec_i^\top\right)\nabla^2f(\bar \x).
\]
Since $\bar \x$ is an index-$k$ critical point and is non-degenerate, the eigenvalues of $\nabla^2f(\bar \x)$ satisfy
\[
\lambda_1\leq\cdots\leq\lambda_k<0<\lambda_{k+1}\leq\cdots\leq\lambda_d.
\]
Therefore, the eigenvalues of $B$ are
\[
-\lambda_1,\ldots,-\lambda_k,\lambda_{k+1},\ldots,\lambda_d,
\]
which are all positive. Hence, $B$ is positive-definite. Therefore, the subproblem
has the unique global minimizer $\s=\mathbf{0}$, which means $\bar \x$ is a fixed point of \Cref{algorithm 1.1}.

Conversely, suppose $\bar \x$ is a fixed point of \Cref{algorithm 1.1}, i.e., $\s=\mathbf{0}$ is a global minimizer of the subproblem. The first-order optimality condition of the subproblem in \Cref{prop:cubic-opt} at $\s=\mathbf{0}$ gives
\[
\left(\I-\sum_{i=1}^k2\vvec_i\vvec_i^\top\right)\nabla f(\bar \x)=\mathbf{0},
\]
and therefore $\nabla f(\bar \x)=\mathbf{0}$. Hence, $\bar \x$ is a critical point of $f$. Let $\bar \x$ be an index-$\bar k$ critical point. Suppose $\bar k\neq k$. If $\bar k>k$, at least one negative eigenvalue of $\nabla^2 f(\bar \x)$ lies outside the reflected subspace. More precisely, $\lambda_{k+1}<0$, and the corresponding eigenvector $\vvec_{k+1}$ satisfies $B\vvec_{k+1}
=
\lambda_{k+1}\vvec_{k+1}$. If $\bar k<k$, then $\lambda_{\bar k+1}>0$,
and the corresponding eigenvector $\vvec_{\bar k+1}$ lies inside the reflected subspace and satisfies $B\vvec_{\bar k+1}
=
-\lambda_{\bar k+1}\vvec_{\bar k+1}$.
In either case,
$\left(\I-\sum_{i=1}^k2\vvec_i\vvec_i^\top\right)\nabla^2f(\bar \x)$
has an eigenpair $(\beta,\vvec)$ with
$\beta<0$ and $\|\vvec\|_2=1$. Since $\nabla f(\bar \x)=\mathbf{0}$, along this direction, we have that 
\[
   q(t\vvec)
   =f(\bar \x)+\frac{\beta}{2}t^2+\frac{\sigma}{3}|t|^3,
\]
where $q$ is the cubic-regularized quadratic function defined in \Cref{prop:cubic-opt}.
For every $0<|t|<-\frac{3\beta}{2\sigma}$, we have $\frac{\beta}{2}t^2+\frac{\sigma}{3}|t|^3<0$,
and hence $q(t\vvec)<f(\bar \x)=q(\mathbf{0})$. 
Therefore, $\s=\mathbf{0}$ is not a global minimizer of the subproblem, which contradicts that $\bar \x$ is a fixed point of \Cref{algorithm 1.1}. Hence, $\bar k=k$, and $\bar \x$ is an index-$k$ critical point of $f$.
\end{proof}

\section{Inexact invariant subspaces and inexact solutions of subproblem}
\label{sec:inexact}

In the actual implementation of \Cref{algorithm 1.1}, $\{\vvec_i\}_{i=1}^k=\mathrm{Eigensolver}(\nabla^2 f(\x_j))$ and the subproblem are usually solved by iterative eigensolvers such as LOBPCG, block Lanczos, or optimization methods for the Rayleigh quotient (also see \Cref{sec: methods for subproblem}). Consequently, we cannot expect to find either the exact unstable eigenspace or the exact solution of the subproblem due to the round-off error and tolerance in these iterative methods. Hence, in this subsection, we allow both of these problems to be solved inexactly. The
error in the unstable eigenspace is measured through spectral projectors,
rather than individual eigenvectors. This metric is invariant under changes
of sign and under orthogonal changes of basis within a multiple eigenspace.

At iteration $j$, let
\[
   H_j=\nabla^2f(\x_j),\quad
   P_j=\sum_{i=1}^k\vvec_i(\x_j)\vvec_i(\x_j)^\top,\quad
   R_j=\I-2P_j.
\]
Since $P_j$ is the spectral projector associated with the first $k$
eigenvalues of $H_j$, we have
\[
   P_jH_j=H_jP_j,
   \quad
   R_jH_j=H_jR_j.
\]
Let $\widetilde V_j\in\mathbb{R}^{d\times k}$ be the inexact solutions of $\{\vvec_i\}_{i=1}^k=\mathrm{Eigensolver}(\nabla^2 f(\x_j))$, which have orthonormal columns and
define
\begin{equation}
   \widetilde P_j=\widetilde V_j\widetilde V_j^\top,\quad
   \widetilde R_j=\I-2\widetilde P_j,\quad
   \delta_j=\|\widetilde P_j-P_j\|_2.
   \label{eq:projector-error}
\end{equation}
The quantity $\delta_j$ is the sine of the largest principal angle between
the exact and computed invariant subspaces. In contrast to $R_jH_j$, the
product $\widetilde R_jH_j$ need not be symmetric because an approximate
spectral projector need not commute with $H_j$. We therefore use its
symmetric part
\begin{equation}
   \widetilde B_j
   =
   \frac{1}{2}
   \left(
   \widetilde R_jH_j+H_j\widetilde R_j
   \right)
   \label{eq:inexact-B}
\end{equation}
in the inexact cubic model
\begin{equation}
   \widetilde q_j(\s)
   =
   f(\x_j)
   +
   \left(\widetilde R_j\nabla f(\x_j)\right)^\top\s
   +
   \frac{1}{2}\s^\top\widetilde B_j\s
   +
   \frac{\sigma_j}{3}\|\s\|_2^3
   \label{eq:inexact-model}
\end{equation}
as in \Cref{algorithm 1.1}. Notice that this symmetrization does not alter the quadratic form since $\s^\top\widetilde R_jH_j\s  =   \s^\top\widetilde B_j\s$.

\begin{assumption}[Inexact computations of the subproblem]
\label{assumption:inexact}
The computed step $\s_j$ satisfies
\begin{equation}
   \boldsymbol{\xi}_j
   =
   \widetilde R_j\nabla f(\x_j)
   +
   \widetilde B_j\s_j
   +
   \sigma_j\|\s_j\|_2\s_j,
   \quad
   \|\boldsymbol{\xi}_j\|_2
   \leq
   \eta_j\|\nabla f(\x_j)\|_2,
   \label{eq:forcing}
\end{equation}
where $\eta_j\geq0$ is a forcing parameter. The approximate invariant
subspace satisfies \eqref{eq:projector-error}. Both $\delta_j$ and $\eta_j$
can be strictly positive but small.
\end{assumption}

The residual condition \eqref{eq:forcing} is a relative accuracy condition
for the inexact cubic step. It permits a nonzero inner residual at every
finite iteration, while forcing its absolute size to decrease as the target
critical point is approached.

\begin{theorem}
\label{thm:inexact-recursion}
Suppose Assumption~\ref{assumption:local} and Assumption~\ref{assumption:inexact} hold. Fix $\bar\eta\geq0$ and $\bar\delta\geq0$ such that $\bar\delta\leq\frac{\mu}{4L}$, and define $K=\frac{2(1+\bar\eta)}{\mu}, \rho=\min\left\{r_0,\frac{\mu}{M}\right\}, r=\frac{\rho}{1+KL}$. Furthermore, define $C_1
   =
   \frac{2L}{\mu}
   \max\left\{1,2LK\right\}, C_2
   =
   \frac{2}{\mu}
   \left(\sigma_{max}+\frac{M}{2}\right)K^2L^2$.
If
$$
   \delta_j\leq\bar\delta,\quad
   \eta_j\leq\bar\eta,\quad
   \|\x_j-\x^*\|_2\leq r,
$$
then the whole line segment
$\{\x_j+t\s_j:t\in[0,1]\}$ is contained in $U$, the approximate reflected
Hessian satisfies $\widetilde B_j\succeq\frac{\mu}{2}\I$,
and
\begin{equation}
   \|\x_{j+1}-\x^*\|_2
   \leq
   C_1(\delta_j+\eta_j)\|\x_j-\x^*\|_2
   +
   C_2\|\x_j-\x^*\|_2^2.
   \label{eq:inexact-recursion}
\end{equation}
\end{theorem}

\begin{proof}
Since $\|\widetilde R_j-R_j\|_2
   =
   2\|\widetilde P_j-P_j\|_2
   =
   2\delta_j$ and $R_jH_j=H_jR_j$, \eqref{eq:inexact-B} gives
\begin{equation}
   \|\widetilde B_j-R_jH_j\|_2=
   \left\|
   \frac{1}{2}
   \left(
   (\widetilde R_j-R_j)H_j
   +
   H_j(\widetilde R_j-R_j)
   \right)
   \right\|_2 \leq
   2L\delta_j.
\label{eq:B-perturbation}
\end{equation}
By Lemma~\ref{lem:positive}, $R_jH_j\succeq\mu\I$. Hence, for every $\z\in\mathbb{R}^d$,
$$
   \z^\top\widetilde B_j\z
   =
   \z^\top R_jH_j\z
   +
   \z^\top(\widetilde B_j-R_jH_j)\z \geq
   \left(\mu-2L\delta_j\right)\|\z\|_2^2.
$$
Using $\delta_j\leq\bar\delta$ and $\bar\delta\leq\frac{\mu}{4L}$, we obtain $\mu-2L\delta_j
   \geq
   \mu-2L\bar\delta
   \geq
   \frac{\mu}{2}$,
which proves $\widetilde B_j\succeq\frac{\mu}{2}\I$. Therefore, the approximate
cubic model \eqref{eq:inexact-model} is strongly convex and has a unique
global minimizer.

Taking the inner product of \eqref{eq:forcing} with $\s$, using
$\widetilde B_j\succeq\frac{\mu}{2}\I$, $\|\widetilde R_j\|_2=1$, and
$\|\boldsymbol{\xi}_j\|_2\leq\eta_j\|\nabla f(\x_j)\|_2$, we get
\[
   \frac{\mu}{2}\|\s_j\|_2^2
   +
   \sigma_j\|\s_j\|_2^3
   \leq
   (1+\eta_j)\|\nabla f(\x_j)\|_2\|\s_j\|_2.
\]
This leads to
\begin{equation}
   \|\s_j\|_2
   \leq
   K\|\nabla f(\x_j)\|_2,
   \quad
   K=\frac{2(1+\bar\eta)}{\mu}.
   \label{eq:inexact-step-bound}
\end{equation}
Since $\nabla f(\x^*)=\mathbf{0}$ and $\x_j\in U$, Assumption~\ref{assumption:local}
gives
\begin{equation}
   \|\nabla f(\x_j)\|_2
   \leq
   L\|\x_j-\x^*\|_2.
   \label{eq:inexact-gradient-upper-bound}
\end{equation}
Therefore, $\|\s_j\|_2
   \leq
   KL\|\x_j-\x^*\|_2$, and for any $t\in[0,1]$, we have that
\begin{equation}
   \|\x_j+t\s_j-\x^*\|_2\leq
   \|\x_j-\x^*\|_2
   +
   t\|\s_j\|_2\leq
   (1+KL)r \leq r_0,
\label{eq:inexact-segment-bound}
\end{equation}
i.e., the whole line segment
$\{\x_j+t\s_j:t\in[0,1]\}$ is contained in $U$. Furthermore,
\begin{equation}
   \|\x_j+\s_j-\x^*\|_2
   \leq
   \rho
   \leq
   \frac{\mu}{M}.
   \label{eq:inexact-new-point-bound}
\end{equation}
Multiplying \eqref{eq:forcing} by $\widetilde R_j$ and using
$\widetilde R_j^2=\I$, we give
\begin{equation}
   \nabla f(\x_j)+H_j\s_j
   =
   \widetilde R_j\boldsymbol{\xi}_j
   +
   \left(H_j-\widetilde R_j\widetilde B_j\right)\s_j
   -
   \sigma_j\|\s_j\|_2\widetilde R_j\s_j.
   \label{eq:inexact-newton-identity}
\end{equation}
From \eqref{eq:inexact-B}, $\widetilde R_j\widetilde B_j
   =
   \frac{1}{2}
   \left(
   H_j+\widetilde R_j H_j\widetilde R_j
   \right)$. Therefore,
\begin{equation}
\begin{aligned}
   \|H_j-\widetilde R_j\widetilde B_j\|_2
   &=
   \frac{1}{2}
   \|H_j-\widetilde R_j H_j\widetilde R_j\|_2 =
   \frac{1}{2}
   \|H_j\widetilde R_j-\widetilde R_j H_j\|_2 \\
   &=
   \frac{1}{2}
   \|H_j(\widetilde R_j-R_j)-(\widetilde R_j-R_j)H_j\|_2 \leq
   2L\delta_j.
\end{aligned}
\label{eq:commutator-bound}
\end{equation}
The Hessian Lipschitz condition and \eqref{eq:inexact-newton-identity} give
\begin{equation}
\begin{aligned}
   \|\nabla f(\x_j+\s_j)\|_2
   &\leq
   \|\boldsymbol{\xi}_j\|_2
   +
   \|H_j-\widetilde R_j\widetilde B_j\|_2\|\s_j\|_2 
   +
   \sigma_j\|\s_j\|_2^2
   +
   \frac{M}{2}\|\s_j\|_2^2 \\
   &\leq
   \eta_j\|\nabla f(\x_j)\|_2
   +
   2L\delta_j\|\s_j\|_2
   +
   \left(\sigma_{max}+\frac{M}{2}\right)\|\s_j\|_2^2.
\end{aligned}
\label{eq:inexact-gradient-bound}
\end{equation}
Similar to the argument leading to \eqref{eq:quadratic-estimate}, by
\eqref{eq:inexact-gradient-bound},
\eqref{eq:inexact-step-bound}, and
\eqref{eq:inexact-gradient-upper-bound}, we obtain
\begin{equation}
\begin{aligned}
   \|\x_{j+1}-\x^*\|_2
   &=
   \|\x_j+\s_j-\x^*\|_2 \leq
   \frac{2}{\mu}
   \|\nabla f(\x_j+\s_j)\|_2 \\
   &\leq
   \frac{2L}{\mu}\eta_j
   \|\x_j-\x^*\|_2
   +
   \frac{4L^2K}{\mu}\delta_j
   \|\x_j-\x^*\|_2 \\
   &\quad
   +
   \frac{2}{\mu}
   \left(\sigma_{max}+\frac{M}{2}\right)
   K^2L^2
   \|\x_j-\x^*\|_2^2 \\
   &\leq
   C_1(\delta_j+\eta_j)
   \|\x_j-\x^*\|_2
   +
   C_2\|\x_j-\x^*\|_2^2,
\end{aligned}
\end{equation}
which proves \eqref{eq:inexact-recursion}.
\end{proof}

\begin{corollary}
\label{cor:inexact-rates}
Assume the same conditions as in Theorem~\ref{thm:inexact-recursion}, and
   \[
      \delta_j+\eta_j\leq\varepsilon,
      \quad \forall j\in\mathbb{N},
      \quad
      C_1\varepsilon<1.
   \]
   Define $
      r_{\rm lin}
      =
      \min\left\{
      r,
      \frac{1-C_1\varepsilon}{2C_2}
      \right\}, q =
      C_1\varepsilon+C_2r_{\rm lin}$. If $\|\x_0-\x^*\|_2\leq r_{\rm lin}$, then
   \[
      \|\x_{j+1}-\x^*\|_2\leq q \|\x_j-\x^*\|_2,
      \quad
      q\leq ({1+C_1\varepsilon})/{2}<1, \  \forall j\in \mathbb{N},
   \]
   and hence $\x_j$ converges linearly to $\x^*$.
 Moreover, if $\delta_j+\eta_j\rightarrow0$, then $\x_j$ converges superlinearly to $\x^*$.
 If, in addition, for all sufficiently large $j$, $\delta_j+\eta_j\leq c \|\x_j-\x^*\|_2$ for some $c>0$, then $\x_j$ converges quadratically to $\x^*$. In particular, if for large $j$, we implement the choices 
 \begin{equation}
   \label{eq:quadratic-forcing}
 \delta_j = O\left(\|\nabla f(\x_j)\|_2\right) \;\text{ and }\; \|\boldsymbol{\xi}_j\|_2 = O\left(\|\nabla f(\x_j)\|_2^2\right)\; (\text{or }\ \eta_j = O\left(\|\nabla f(\x_j)\|_2\right)), 
 \end{equation}
then the  quadratic convergence holds.
\end{corollary}

\begin{proof}
Whenever $\|\x_j-\x^*\|_2\leq r_{\rm lin}$,
Theorem~\ref{thm:inexact-recursion} gives
\begin{equation}
   \|\x_{j+1}-\x^*\|_2
   \leq
   C_1\varepsilon \|\x_j-\x^*\|_2+C_2\|\x_j-\x^*\|_2^2 \leq q\|\x_j-\x^*\|_2.
   \label{eq: linear}
\end{equation}
By $q
   \leq
   ({1+C_1\varepsilon})/{2}
   <1$, we have that $\|\x_{j+1}-\x^*\|_2 \leq \|\x_j-\x^*\|_2\leq r_{\rm lin}$. 
   
   Since $\|\x_0-\x^*\|_2\leq r_{\rm lin}$, an induction argument gives
$\|\x_{j+1}-\x^*\|_2\leq q \|\x_j-\x^*\|_2$ for every $j\in\mathbb{N}$, i.e., $\x_j$ converges $Q$-linearly to $\x^*$. 

Moreover, since $\limsup_{j\rightarrow \infty}\frac{\|\x_{j+1}-\x^*\|_2}{\|\x_j-\x^*\|_2}\leq C_1\epsilon$, a smaller $\epsilon$ that corresponds to solving both subproblems more accurately leads to a smaller upper bound on the asymptotic convergence factor, and hence, potentially faster local convergence.

If, in addition, $\delta_j+\eta_j\rightarrow0$, dividing \eqref{eq:inexact-recursion} by $\|\x_j-\x^*\|_2$ gives \[\frac{\|\x_{j+1}-\x^*\|_2}{\|\x_j-\x^*\|_2}
   \leq
   C_1(\delta_j+\eta_j)
   +
   C_2\|\x_j-\x^*\|_2.
\]
The linear convergence rate in \eqref{eq: linear} gives $\|\x_j-\x^*\|_2\rightarrow0$,  hence $\frac{\|\x_{j+1}-\x^*\|_2}{\|\x_j-\x^*\|_2}\rightarrow0$. Therefore, the convergence is superlinear. If, in addition, for all sufficiently large $j$, $\delta_j+\eta_j\leq c \|\x_j-\x^*\|_2$ for some $c>0$. We can similarly get that $\x_j$ converges quadratically to $\x^*$. Furthermore, for every $\x_j\in U$ satisfying $\|\x_j-\x^*\|_2
   \leq
   \frac{\mu}{M}$, the preceding estimates give the explicit local equivalence 
   \[\frac{\mu}{2}\|\x_j-\x^*\|_2
   \leq
   \|\nabla f(\x_j)\|_2
   \leq
   L\|\x_j-\x^*\|_2.\]
   Hence, the readily implementable conditions \eqref{eq:quadratic-forcing}
imply $\delta_j+\eta_j
   =
   O\left(\|\nabla f(\x_j)\|_2\right)
   =
   O\left(\|\x_j-\x^*\|_2\right)$,
and therefore preserve the quadratic convergence.
\end{proof}

\begin{remark}
\label{rem:practical-forcing}
Both errors in
\eqref{eq:quadratic-forcing} may be strictly positive at every finite
iteration. Fixed sufficiently small bounds on $\delta_j$ and on the relative
residual $\eta_j$ preserve local convergence, but generally at a linear
rate. A fixed nonzero absolute residual of the subproblem does not give the same convergence conclusion. If
\eqref{eq:forcing} is replaced by $\|\boldsymbol{\xi}_j\|_2
   \leq
   \varepsilon_{\rm s}$, then \eqref{eq:inexact-gradient-bound} contains the additive term
$\varepsilon_{\rm s}$, and the corresponding local estimate contains an
additive term. Therefore, in general one can guarantee convergence only to an
$O(\varepsilon_s)$ neighborhood of $\x^*$.
\end{remark}

\begin{proposition}
\label{prop:inexact-stability}
Fix $k\in\{0,\ldots,d\}$ in \Cref{algorithm 1.1}. Suppose all the critical points of $f$ are non-degenerate and Assumption~\ref{assumption:inexact} holds with $0\leq \eta_j\leq\bar\eta<1$. In addition, suppose the computed step $\s_j$ satisfies the second-order condition
\begin{equation}
   \lambda_{\min}
   \left(
   \widetilde B_j+\sigma_j\|\s_j\|_2\I
   \right)
   \geq
   -\omega_j,
   \quad
   \omega_j\geq0.
   \label{eq:inexact-second-order}
\end{equation}
Suppose that the accuracy of the approximate invariant subspace $\delta_j$ and the second-order accuracy of the subproblem $\omega_j$  satisfy
\begin{equation}
2\|\nabla^2f(\bar \x)\|_2\delta_j+\omega_j   <  \min_{1\leq i\leq d}   \left| \lambda_i\left(\nabla^2f(\bar \x)\right)   \right|:= \gamma(\bar \x),\; \forall \text{ critical point } \bar \x.
\label{eq:inexact-index-condition}
\end{equation}
Then, $\s_j=\mathbf{0}$ satisfies both \eqref{eq:forcing} and \eqref{eq:inexact-second-order} at $\x_j=\bar \x$ if and only if $\bar \x$ is an index-$k$ critical point of $f$.
\end{proposition}

\begin{proof}
Suppose first that $\bar \x$ is an index-$k$ critical point of $f$. Since
$\nabla f(\bar \x)=\mathbf{0}$, $\|\boldsymbol{\xi}_j\|_2  =  0  \leq   \eta_j\|\nabla f(\bar \x)\|_2$,
and therefore the first-order condition \eqref{eq:forcing} is satisfied.
Let
$$H=\nabla^2f(\bar \x),\quad
   P=\sum_{i=1}^k\vvec_i\vvec_i^\top,\quad
   R=\I-2P, \quad B=RH.$$
Since $\bar \x$ is an index-$k$ critical point, the eigenvalues of $H$ satisfy
\[
   \lambda_1\leq\cdots\leq\lambda_k<0  <  \lambda_{k+1}\leq\cdots\leq\lambda_d.
\]
Therefore, the eigenvalues of $B$ are $-\lambda_1,\ldots,-\lambda_k,
   \lambda_{k+1},\ldots,\lambda_d$,
hence, $B\succeq\gamma(\bar \x)\I$.
Since $\|\widetilde R_j-R\|_2 =  2\|\widetilde P_j-P\|_2  =
   2\delta_j$ and $RH=HR$, the definition of $\widetilde B_j$ gives
\begin{equation}\label{eq:inexact-B-critical-perturbation}
   \|\widetilde B_j-B\|_2 =
   \left\|  \frac{1}{2}   \left(   (\widetilde R_j-R)H
   +  H(\widetilde R_j-R)  \right)
   \right\|_2 \leq   2\|H\|_2\delta_j.
\end{equation}
It follows from $B\succeq\gamma(\bar \x)\I$ and
\eqref{eq:inexact-B-critical-perturbation} that $\lambda_{\min}(\widetilde B_j) \geq   \gamma(\bar \x) -  2\|H\|_2\delta_j.$ By \eqref{eq:inexact-index-condition},
\[
   \gamma(\bar \x) -   2\|H\|_2\delta_j  >   \omega_j  \geq -\omega_j.
\]
Therefore, $\lambda_{\min} \left( \widetilde B_j+\sigma_j\|\mathbf{0}\|_2\I  \right) =  \lambda_{\min}(\widetilde B_j)  \geq -\omega_j$, 
and \eqref{eq:inexact-second-order} is satisfied. Hence, $\bar \x$ admits the zero inexact step, that is, $\s_j=\mathbf{0}$.

Conversely, suppose $\bar \x$ admits the zero inexact step. Taking
$\s_j=\mathbf{0}$ in \eqref{eq:forcing} gives $\boldsymbol{\xi}_j
   =
   \widetilde R_j\nabla f(\bar \x)$. Since $\widetilde R_j$ is an orthogonal reflection, $\|\widetilde R_j\nabla f(\bar \x)\|_2
   =
   \|\nabla f(\bar \x)\|_2$. Therefore, \eqref{eq:forcing} yields $\|\nabla f(\bar \x)\|_2
   \leq
   \eta_j\|\nabla f(\bar \x)\|_2$. Since $\eta_j\leq\bar\eta<1$, it follows that $\nabla f(\bar \x)=\mathbf{0}$. Hence, $\bar \x$ is a critical point of $f$.  
   Let  $\bar k$ denote the index of $\bar \x$. Suppose $\bar k\neq k$. If $\bar k>k$, then at least one negative eigenvalue of $\nabla^2f(\bar \x)$ lies outside the reflected subspace. If $\bar k<k$, then at least one positive eigenvalue lies inside the reflected subspace and is changed to a negative one. In either case, the exact reflected Hessian has a negative eigenvalue satisfying
$  \lambda_{\min}(B)  \leq   -\gamma(\bar \x)$.
Using \eqref{eq:inexact-B-critical-perturbation} and
$  \lambda_{\min}(B)  \leq
   -\gamma(\bar \x)$, we obtain
$$
   \lambda_{\min}(\widetilde B_j)  \leq
   \lambda_{\min}(B)   +
   \|\widetilde B_j-B\|_2 
   \leq   -\gamma(\bar \x)   +   2\|\nabla^2f(\bar \x)\|_2\delta_j.$$
By \eqref{eq:inexact-index-condition}, $-\gamma(\bar \x)  +
   2\|\nabla^2f(\bar \x)\|_2\delta_j < -\omega_j$. Hence, $\lambda_{\min}(\widetilde B_j)   <   -\omega_j$. On the other hand, since $\bar \x$ admits the zero inexact step, we have
\eqref{eq:inexact-second-order} with $\s_j=\mathbf{0}$, which requires $\lambda_{\min}(\widetilde B_j)
   \geq
   -\omega_j$. We thus have a contradiction. Consequently, we must have $\bar k=k$, and $\bar \x$ is an
index-$k$ critical point of $f$.
\end{proof}

\section{The methods for updating the invariant subspaces and solving the subproblems}\label{sec: methods for subproblem}
A key advantage of \Cref{algorithm 1.1} is its flexibility in the choice of numerical procedures for solving the subproblems and computing the invariant subspaces for $\{\vvec_i\}_{i=1}^k=\mathrm{Eigensolver}(\nabla^2 f(\x_j))$. In particular, since consecutive iterates $\x_j$ and $\x_{j+1}$ are typically close, the negative-curvature directions $\{\vvec_i\}_{i=1}^k$ of $\nabla^2 f(\x_j)$ computed at the $j$-th iteration can provide effective initial guesses for the corresponding directions of $\nabla^2 f(\x_{j+1})$. Hence, in practice, these directions can be updated rather than recomputed from scratch, which keeps the associated computational cost moderate. Moreover, the cubic subproblem has the same form as the standard cubic regularization subproblem. Consequently, one can directly exploit efficient numerical techniques developed for cubic-regularized minimization. Therefore, the cost of solving the subproblem is also moderate.

\subsection{Solvers for the cubic subproblem}

We first briefly introduce numerical approaches for solving the cubic subproblem in \Cref{algorithm 1.1}: \begin{equation} \min_{\s\in\mathbb{R}^d} \; q(\s):= f(\x)+ \g^\top \s+\frac12 \s^\top H\s+\frac{\sigma}{3}\|\s\|^3_2, \quad \sigma>0, \label{eq:cubic_subproblem} \end{equation} where $H\in\mathbb{R}^{d\times d}$ is symmetric. By \Cref{prop:cubic-opt}, a vector $\s^*$ is a global minimizer of \eqref{eq:cubic_subproblem} if and only if there exists $\lambda^*\geq 0$ such that \begin{equation} (H+\lambda^* \I)\s^*=-\g, \quad H+\lambda^* \I\succeq 0, \quad \lambda^*=\sigma\|\s^*\|_2. \label{eq:cubic_opt_cond} \end{equation} These conditions reduce the nonlinear vector optimization problem to the determination of a scalar shift $\lambda^*$ together with a shifted linear system; see, e.g., \cite{nesterov2006cubic,cartis2011adaptive}.

\paragraph{Direct solution for small- and medium-scale problems}
When $d$ is moderate and an explicit factorization of $H$ is affordable, \eqref{eq:cubic_subproblem} can be solved conveniently through a spectral decomposition. Let \[ H=U\Theta U^\top, \quad \Theta=\operatorname{diag}(\theta_1,\ldots,\theta_d), \quad \theta_1\leq\cdots\leq\theta_d, \] and define $\widehat \g=U^\top \g$. In the nondegenerate case $\lambda^*>\max\{0,-\theta_1\}$, \eqref{eq:cubic_opt_cond} gives \[ \s^* = -U(\Theta+\lambda^* \I)^{-1}\widehat \g, \] where $\lambda^*\in (\max\{0,-\theta_1\},\infty)$ is the unique solution of the scalar equation \begin{equation} \lambda^2 = \sigma^2 \sum_{i=1}^d \frac{\widehat g_i^2}{(\theta_i+\lambda)^2}. \label{eq:cubic_secular} \end{equation} Thus, after an eigendecomposition of $H$, the original $d$-dimensional problem reduces essentially to a one-dimensional root-finding problem, which can be treated by standard root solvers. The only additional care is required in the so-called hard case. If $\g$ is orthogonal to the eigenspace associated with $\theta_1$ and the optimal shift satisfies $\lambda^*=-\theta_1$, the shifted matrix is singular and the solution has the form \begin{equation} \s^* = -(H-\theta_1 \I)^\dagger \g+\mathbf{u}, \quad \mathbf{u}\in\ker(H-\theta_1 \I), \label{eq:cubic_hardcase} \end{equation} where the norm of $\mathbf{u}$ is chosen so that $\sigma\|\s^*\|_2=-\theta_1$. Consequently, for low-dimensional problems, a spectral decomposition provides a conceptually simple and numerically direct solver for both the generic and hard cases. 

\paragraph{Lanczos--Krylov subspace methods for large-scale problems} For large-scale problems, forming a full spectral decomposition is very expensive. A standard alternative \cite{gould1999solving} is to restrict \eqref{eq:cubic_subproblem} to a Krylov subspace generated by $H$ and $\g$, \[ \mathcal K_\ell(H,\g) := \operatorname{span}\{\g,H\g,\ldots,H^{\ell-1}\g\}, \quad \ell\geq 2. \] Let $Q_\ell$ be an orthonormal basis generated by the Lanczos process. Then $$ H Q_\ell = Q_\ell T_\ell + \beta_\ell \mathbf q_{\ell+1}\mathbf e_\ell^\top, \quad Q_\ell^\top \g=\|\g\|_2\mathbf e_1,  $$ where $T_\ell=Q_\ell^\top H Q_\ell$ is symmetric tridiagonal. Writing $\s=Q_\ell \y$ yields the reduced problem 
\begin{equation} \y_\ell \in \arg \min_{\y\in\mathbb{R}^{\ell}} \left\{ f(\x)+ \|\g\|_2 \mathbf e_1^\top \y +\frac12 \y^\top T_\ell \y +\frac{\sigma}{3}\|\y\|^3_2 \right\}, \quad \s_\ell=Q_\ell \y_\ell. \label{eq:reduced_cubic} \end{equation} Since $\ell\ll d$ in typical large-scale applications, \eqref{eq:reduced_cubic} can in turn be solved by the direct spectral procedure described in \eqref{eq:cubic_secular}. The resulting scheme requires only matrix--vector products with $H$ and is therefore particularly applicable when the Hessian is either sparse (e.g., for discrete energy functionals with local interactions) or available only through Hessian--vector products. For completeness, we recall the basic convergence behavior in the nondegenerate case \cite{carmon2018analysis,jia2022solving}. Suppose that $\g$ has a nonzero component in the eigenspace corresponding to $\theta_1$, and let \[ H^*:=H+\lambda^* \I, \quad \kappa := \frac{\theta_d+\lambda^*} {\theta_1+\lambda^*}, \quad \rho := \frac{\sqrt{\kappa}-1} {\sqrt{\kappa}+1} <1. \] Then the Lanczos approximations satisfy, up to constants independent of $\ell$, \begin{equation} \|\s_\ell-\s^*\|_2 = \mathcal O(\rho^\ell), \quad q(\s_\ell)-q(\s^*) = \mathcal O(\rho^{2\ell}). \label{eq:lanczos_cubic_rate} \end{equation}
Moreover, Lanczos terminates with the exact solution once the Krylov subspace reaches the grade of $\g$ with respect to $H$, i.e., once the subspace ceases to grow. The hard case requires special treatment. Indeed, if $\g$ is orthogonal to the leftmost eigenspace of $H$, then $\mathcal K_\ell(H,\g)$ remains orthogonal to this eigenspace for every $\ell$, whereas the global minimizer may contain a nonzero component in that direction. A restart with an additional direction is therefore needed to recover the missing eigenspace component; see \cite{gould1999solving,carmon2018analysis,jia2022solving}. 

\paragraph{Variants and alternative solvers} Some variants of the basic Krylov approach have been developed to
address limitations that arise in large-scale cubic subproblems. For example, for ill-conditioned problems, the dimension of the subspace required
for an accurate solution can become large, leading to increased storage
and orthogonalization costs. Nested restarted Lanczos methods address
this issue by periodically restarting the Krylov process while retaining
a small set of useful correction directions
\cite{zhang2018nested}. Alternatively, the cubic subproblem can be reformulated as a generalized eigenvalue problem \cite{lieder2020solving}, while first-order methods such as gradient descent or accelerated projected methods have also been developed \cite{carmon2019gradient,jiang2021accelerated}.

\subsection{Updating the invariant subspace}
\label{subsec:update_invariant_subspaces}
In this subsection, we introduce some methods to realize $\{\vvec_i\}_{i=1}^k=\mathrm{Eigensolver}(\nabla^2 f(\x_j))$ in \Cref{algorithm 1.1}. Let $H_j := \nabla^2 f(\x_j)$, and suppose that
$$
    V_j=[\vvec_1^{(j)},\ldots,\vvec_k^{(j)}]\in\mathbb{R}^{d\times k},
    \quad
    V_j^\top V_j=\I,
$$
approximates the invariant subspace of $H_j$ associated with its $k$
smallest eigenvalues. Since consecutive iterates $\x_j$ and $\x_{j+1}$
are typically close, the corresponding Hessians are also expected to
vary moderately. Consequently, the invariant subspace computed at the
$j$th iteration provides a natural initial approximation for the
corresponding invariant subspace of $H_{j+1}$. Rather than recomputing
the negative-curvature directions from scratch, we therefore update
them using an iterative eigensolver initialized with $V_j$. Similar ideas have been used for sequences of related eigenvector problems \cite{zbikowski2023bootstrapped, salas2015spectral, Zhuzhunashvili2017Preconditioned, 2019High}. We describe
two representative choices below.

\paragraph{Block Lanczos--Krylov methods}
A natural approach is to construct a block Krylov subspace
initialized by the previously computed invariant subspace,
\begin{equation}
    \mathcal K_{\ell}(H_{j+1},V_j)
    :=
    \operatorname{span}
    \left\{
        V_j,\,
        H_{j+1}V_j,\,
        \ldots,\,
        H_{j+1}^{\ell-1}V_j
    \right\}.
    \label{eq:block_krylov_invariant}
\end{equation}
The block Lanczos process generates an orthonormal basis of
\eqref{eq:block_krylov_invariant} while exploiting the symmetry of
$H_{j+1}$. A Rayleigh--Ritz procedure over the resulting subspace then
provides approximations to the leftmost eigenpairs of $H_{j+1}$.
In particular, if $Q_\ell$ denotes an orthonormal basis of the block
Krylov space, the projected eigenvalue problem $Q_\ell^\top H_{j+1}Q_\ell y=\theta y$
is solved, and the corresponding Ritz vector $Q_\ell y$ is used to
approximate an eigenvector of $H_{j+1}$. It is especially advantageous to use $V_j$ as the starting block in
the present setting. When the Hessian changes only slightly between two
consecutive outer iterations, $V_j$ already contains substantial
information about the desired eigenspace of $H_{j+1}$, and therefore
only a moderate Krylov expansion may be required. Moreover, the block
form is preferable to a single-vector Lanczos iteration when several
negative eigenvalues are clustered or have nontrivial multiplicities.

For large-scale problems, allowing the Krylov space to grow indefinitely
is undesirable because of both storage and orthogonalization costs.
In our setting, however, the Krylov space is not accumulated across
outer iterations. At iteration $j$, a Krylov subspace of a fixed depth $L$
is constructed for the current matrix $\nabla^2 f(\x_j)$, using the approximate
invariant subspace $V_{j-1}$ obtained from the previous iteration as the
initial block:
\begin{equation}
    \mathcal{K}_L(\nabla^2 f(\x_j),V_{j-1})
    =
    \operatorname{span}
    \left\{
        V_{j-1},
        \nabla^2 f(\x_j)V_{j-1},
        \ldots,
        \nabla^2 f(\x_j)^{L-1}V_{j-1}
    \right\}.
\end{equation}
A Rayleigh--Ritz extraction is then performed over this subspace of fixed dimension 
 to obtain the updated approximation $V_t$ associated with the
leftmost Ritz values. The remaining Krylov basis vectors are discarded,
and only $V_t$ is recycled as the initial block for the next matrix
$A_{t+1}$. Thus, the dimension of the Krylov subspace, and consequently
the storage and orthogonalization costs, remain bounded at every outer
iteration. This construction can be viewed as a fixed-depth, warm-started block Krylov iteration.

\paragraph{Restarted LOBPCG}
The locally optimal block
preconditioned conjugate gradient (LOBPCG) method
\cite{knyazev2001toward} provides another effective approach for
updating the desired invariant subspace. Starting from $X^{(0)}=V_j$, LOBPCG maintains a block approximation
$X^{(t)}\in\mathbb{R}^{d\times k}$ to the invariant subspace associated
with the $k$ smallest eigenvalues of $H_{j+1}$. Let
\begin{equation}
    \Theta^{(t)}
    :=
    (X^{(t)})^\top H_{j+1}X^{(t)}
    \label{eq:lobpcg_rayleigh}
\end{equation}
be the block Rayleigh quotient. The corresponding residual is
\begin{equation}
    R^{(t)}
    :=
    H_{j+1}X^{(t)}
    -
    X^{(t)}\Theta^{(t)}.
    \label{eq:lobpcg_residual2}
\end{equation}
Given a preconditioner $T^{(t)}$, the residual is transformed into
\[
    W^{(t)}=T^{(t)}R^{(t)},
\]
and the next approximation is obtained by performing a Rayleigh--Ritz
procedure over the locally optimal trial subspace
\begin{equation}
    \mathcal S^{(t)}
    :=
    \operatorname{span}
    \left\{
        X^{(t)},\,
        W^{(t)},\,
        P^{(t)}
    \right\},
    \label{eq:lobpcg_trial_space}
\end{equation}
where $P^{(t)}$ contains search directions inherited from the preceding
LOBPCG iteration. The Ritz vectors associated with the $k$ smallest
Ritz values are retained to form $X^{(t+1)}$. The locally optimal construction in
\eqref{eq:lobpcg_trial_space} combines three sources of information:
the current approximation, the preconditioned residual, and the search
directions accumulated from the preceding iteration. Consequently,
LOBPCG can often obtain accurate approximations using a small trial
subspace, without explicitly constructing a long Krylov basis.
Its block structure is also advantageous when the targeted eigenvalues
are clustered.

Both approaches, as well as other eigensolvers, fit naturally within \Cref{algorithm 1.1} and can be warm-started using the approximate invariant subspace obtained from the preceding outer iteration. Importantly, the outer framework does not depend on the particular eigensolver used and, as discussed in \Cref{sec:inexact}, does not require the exact eigenspace. It only requires a sufficiently accurate approximation of the invariant subspace associated with the relevant $k$ smallest eigenvalues.
\subsection{Adaptive strategy for the cubic regularization parameter}
Finally, the remaining detail for implementing \Cref{algorithm 1.1} is the choice of the cubic regularization parameter $\sigma_j$. An excessively large $\sigma_j$ may lead to overly conservative steps, thus slowing down practical convergence. On the other hand, an excessively small $\sigma_j$ can also be undesirable for two reasons. First, if the cubic regularization is insufficient to control the third-order remainder of $f$, the surrogate model in the subproblem may provide a poor approximation. Second, a small $\sigma_j$ may allow a relatively large step, leading to a substantial variation in the Hessian between consecutive iterates. Consequently, the associated invariant subspaces may change significantly, which can make the Eigensolver procedure less robust. 

In standard trust-region and cubic regularization methods for minimization, the regularization parameter is typically adjusted by comparing the actual reduction in the objective with the reduction predicted by the local model \cite{conn2000trust, cartis2011adaptive}. This criterion is not directly suitable for saddle-point search, since the iteration is designed to decrease the energy along the stable subspace while increasing it along the unstable subspace. Consequently, the total energy variation need not be negative, and a reduction-based ratio may not provide an appropriate measure of the quality of the local model. Instead, we measure the accuracy of the second-order Taylor approximation, independently of the sign of the energy variation. Let $\s_j=\x_{j+1}-\x_j$ denote the step at iteration $j$, and define $\Delta_j=\|
\nabla f(\x_{j+1})-\nabla f(\x_j)
-\nabla^2 f(\x_j)\s_j\|_2$. Since the remainder of the Taylor expansion is of order $O(\|\s_j\|^2_2)$, we compare this discrepancy with the cubic regularization term and define $d_j
=
\frac{\Delta_j}
{\sigma_j\|\s_j\|_2^2}$. Thus, $d_j$ measures the size of the observed higher-order model error relative to the current cubic regularization. The parameter $\sigma_j$ is then adaptively updated according to
\begin{equation}\label{eq: adaptive parameter}
\sigma_{j+1}=
\begin{cases}
\max(\gamma_{dec}\sigma_j,\sigma_{min}), \ d_j < d_l, \\
\sigma_j, \ d_l\le d_j \le d_u,\\
\min(\gamma_{inc}\sigma_j,\sigma_{max}), \ d_j > d_u,
\end{cases}
\end{equation}
where $0<\gamma_{dec}<1, \gamma_{inc}>1$, $0<d_l<d_u$ are constants (e.g. $\gamma_{dec}=0.6, \gamma_{inc}=2$, $d_l=0.5$, $d_u=1$). When $d_j$ is small, the cubic term is relatively conservative compared with the observed model discrepancy, and $\sigma_j$ is reduced. When $d_j$ lies in the acceptance interval, the current regularization level is retained. If $d_j>d_u$, the higher-order model error exceeds the scale of the cubic regularization, and $\sigma_j$ is increased to provide stronger stabilization. The lower and upper bounds, $\sigma_{min}$ and $\sigma_{max}$, prevent the regularization parameter from becoming excessively small or large during the iteration. As an additional practical safeguard, one may optionally reject a trial step when $d_j$ exceeds a prescribed threshold and recompute the step with an increased value of $\sigma_j$. It is worth noting that computing $\Delta_j$ incurs essentially no additional computational cost, since $\nabla f(\x_j),\nabla f(\x_{j+1})$ and $\nabla^2f(\x_j)$ are already available during the iteration.

\section{Construction of the solution landscape}\label{sec: solution landscape}
Starting from the same initial condition, \Cref{algorithm 1.1} with different $k$ enables us to find several critical points with various indices of $f$. How can we search for as many critical points as possible from these existing critical points? Constructing the solution landscape is a way to answer this question \cite{yin2021searching,yin2020construction}. 

Using \Cref{algorithm 1.1}, we construct the solution landscape by two procedures: the
downward search that enables us to search for lower-index critical points; the upward search to find the higher-index critical points.
First, the downward search is used to search for possible lower-index critical points
starting from an existing high-index critical point by following its unstable directions. We assume that an index-$k$ critical point $\x^*$ and eigenvectors $\{\vvec_i\}_{i=1}^k$ are provided by \Cref{algorithm 1.1}. To search for a new lower index-$\bar k$ ($\bar k<k$) critical point, we choose an unstable
direction $\vvec_{ini}$ (a typical choice is $\vvec_{m+1}$) which is a linear combination of $\{\vvec_i\}_{i=1}^k$ as the perturbation direction
of $\x^*$, and set $\x^*\pm \bar \epsilon \vvec_{ini}$ as an initial condition of \Cref{algorithm 1.1} of index $m$. Since local attractors of \Cref{algorithm 1.1} of index $m$ are index-$m$ critical points, \Cref{algorithm 1.1} typically drives the iteration point towards index-$m$ critical points, which differ from the existing index-$k$ critical points. Normally, we can find a pair of index-$m$ critical
points, which correspond to the positive and negative driving forces. This also highlights the difference between our method and Newton’s method. If we apply Newton’s method starting from $\x^*\pm \bar \epsilon \vvec_{ini}$, the iteration will be driven back toward $\x^*$, rather than converging to a new index-$m$ critical point, because every nondegenerate critical point is a local attractor of Newton’s method. Similarly, starting from a known lower-index critical point, we can try to search for higher-index critical points. By repeating the downward search and upward search, we can systematically search for possible critical points, including both unstable saddle points of different indices and minima. See \Cref{subsec: four-well potential} and \Cref{subsec: LdG} for examples.
\section{Numerical experiments}\label{sec: Numerics}
In this section, we present several numerical experiments to validate our theoretical results and demonstrate the algorithm's applicability. In particular, we demonstrate the quadratic convergence rate with “exact” subproblem solutions (\Cref{thm:quadratic}), the index-selectivity properties in \Cref{thm:stability} and \Cref{prop:inexact-stability}, and the linear convergence rate with inexact computations (\Cref{thm:inexact-recursion}). We also illustrate the construction of solution landscapes, demonstrate the advantage of the adaptive regularization strategy over fixed regularization parameters, and compare it with the standard adaptive strategy used in optimization.

\subsection{Four-well potential}\label{subsec: four-well potential}
\begin{figure}
    \centering
    \includegraphics[width=0.99\linewidth]{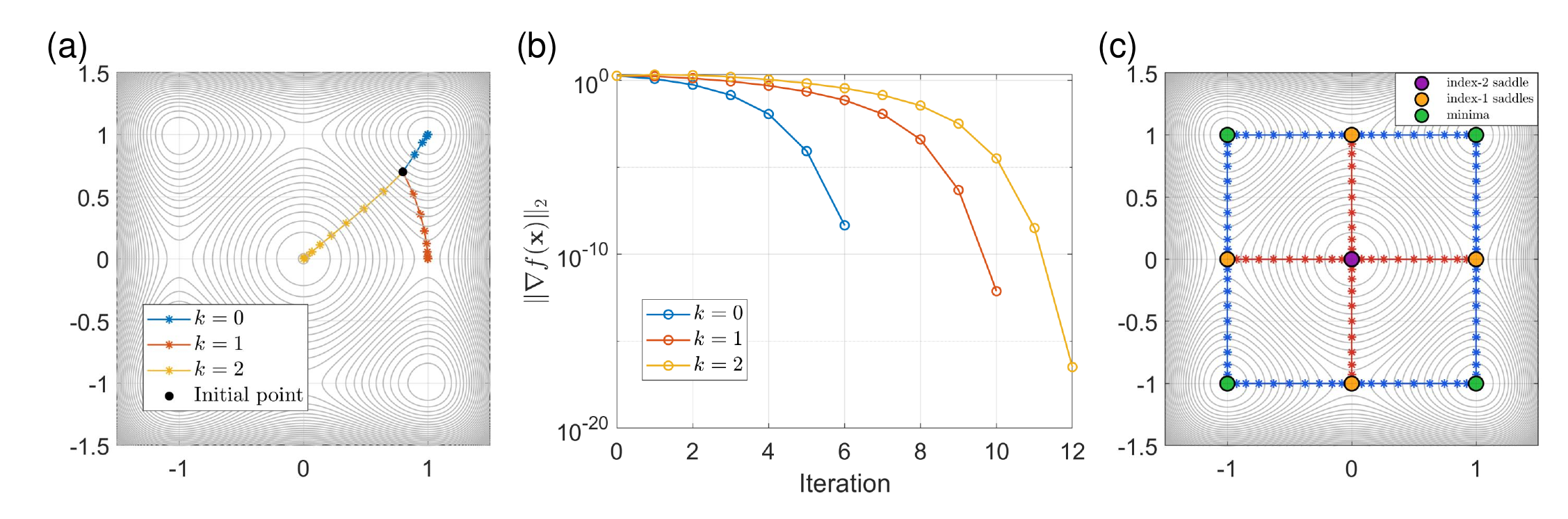}
    \caption{
Numerical results for the four-well potential \eqref{eq: double well potential}.
(a) Trajectories of \Cref{algorithm 1.1} starting from the same initial point $(0.8,0.7)$ with $k=0,1,2$, which converge to the minimum $(1,1)$, the index-$1$ saddle point $(1,0)$, and the index-$2$ saddle point $(0,0)$, respectively.
(b) Convergence histories of $\|\nabla f(\x_j)\|_2$ for the three trajectories in (a), demonstrating quadratic convergence.
(c) Construction of the complete solution landscape starting from the index-$2$ saddle point $(0,0)$. The warm-red trajectories connect the index-$2$ saddle point to the four index-$1$ saddle points, while the blue trajectories connect the index-$1$ saddle points to the four minima.}
    \label{fig:1}
\end{figure}

In \Cref{fig:1}, we apply \Cref{algorithm 1.1} with different values of $k$, using the same initial condition, to the quartic four-well potential
\begin{equation}
f(x,y) = (x^2-1)^2 + (y^2-1)^2,
\label{eq: double well potential}
\end{equation}
which has four minimizers $(\pm 1,\pm 1)$, four index-$1$ saddle points $(0,\pm 1)$ and $(\pm 1,0)$, and one index-$2$ saddle point $(0,0)$. In all computations in this example, we use the fixed cubic regularization parameter $\sigma=50$. Since the problem is two-dimensional, the invariant subspace is computed by a full eigendecomposition of the Hessian, and the cubic subproblem is solved to numerical precision.

In \Cref{fig:1}(a), starting from the initial condition $(0.8,0.7)$, the trajectories corresponding to $k=0,1,2$ converge to the minimum $(1,1)$, the index-$1$ saddle point $(1,0)$, and the index-$2$ saddle point $(0,0)$, respectively, in agreement with \Cref{thm:stability}. In \Cref{fig:1}(b), we plot the error $\|\nabla f(\x)\|_2$ against the iteration number for these three trajectories, demonstrating the quadratic convergence rate. Thus, by varying $k$ while keeping the same initial condition $(0.8,0.7)$, \Cref{algorithm 1.1} identifies critical points of different indices, including the index-$2$ saddle point $(0,0)$.

In \Cref{fig:1}(c), we illustrate how the solution landscape can be constructed from this index-$2$ saddle point to locate all critical points of \eqref{eq: double well potential}. Since $\nabla^2 f(0,0)$ has two unstable directions, $\vvec_1=(0,1)$ and $\vvec_2=(1,0)$, we generate four perturbed initial conditions along $\pm \vvec_i$, $i=1,2$. Starting from these initial conditions, \Cref{algorithm 1.1} with $k=1$ produces four trajectories that converge to the four index-$1$ saddle points $(0,\pm 1)$ and $(\pm 1,0)$, as shown by the warm-red trajectories in \Cref{fig:1}(c). Subsequently, by perturbing each of these index-$1$ saddle points along its unstable direction and applying \Cref{algorithm 1.1} with $k=0$, we obtain the remaining four minima, as indicated by the blue trajectories in \Cref{fig:1}(c). In this way, the complete solution landscape of the four-well potential can be systematically constructed from the index-$2$ saddle point.

\subsection{Rosenbrock function}

\begin{figure}
    \centering
    \includegraphics[width=0.9\linewidth]{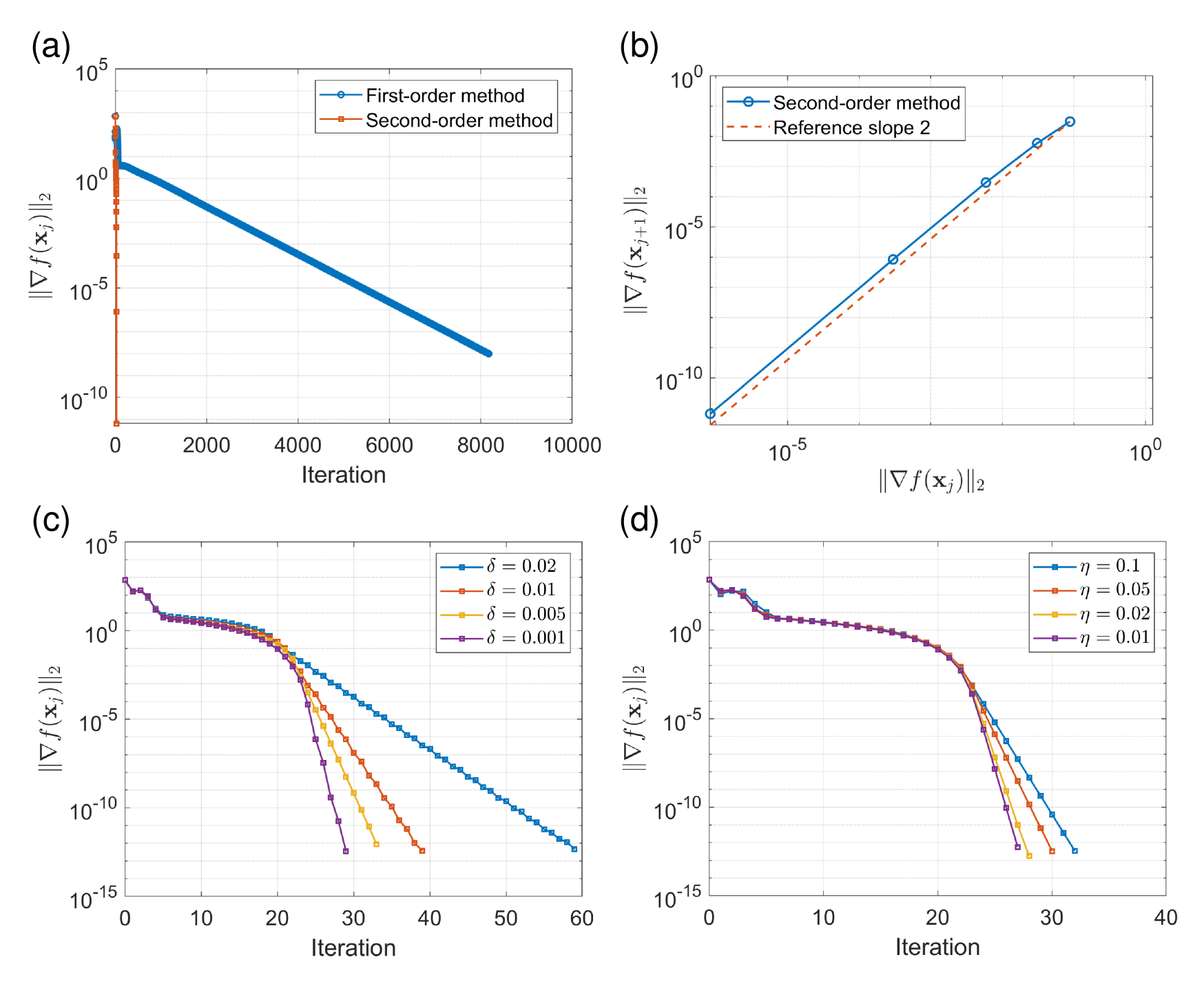}
    \caption{
Numerical results for the modified Rosenbrock function.
(a) Comparison between the first-order method \eqref{eq: first-order method} and the second-order method in \Cref{algorithm 1.1}, measured by the decay of $\|\nabla f(\x_j)\|_2$ with respect to the iteration number.
(b) Log-log plot of $\|\nabla f(\x_{j+1})\|_2$ versus $\|\nabla f(\x_j)\|_2$ for the second-order method. The reference line has slope $2$, illustrating the quadratic convergence rate.
(c)--(d) Convergence histories of \Cref{algorithm 1.1} with inexact computations. In (c), $\delta$ characterizes the error level in the unstable subspace, while in (d), $\eta$ characterizes the accuracy of the numerical solution to the cubic subproblem.}
    \label{fig:2}
\end{figure}

The Rosenbrock function \cite{shang2006note} features a narrow and curved valley with a flat energy landscape around the minimum, so it is also known as the banana function. It is a classical benchmark for optimization algorithms due to its challenging geometry, which makes convergence to the global minimum difficult. To adapt the Rosenbrock example for saddle-point computation, we modify it by adding extra quadratic arctangent terms \cite{2019High}, yielding
$$
  f(\x)=\sum_{i=1}^{d-1}100(x_{i+1}-x_i^2)^2+(1-x_i)^2+\sum_{i=1}^ds_i\arctan^2(x_i-1),
$$
where $\{s_i\}_{i=1}^d$ are the parameters introduced to tune the index of the critical point $\x^*=(1,\cdots,1)^\top$. We take $d=100, s_1=-1000,s_i=1, i\neq 1$, $\x^*$ is an index-1 saddle point. In all experiments in this example, we use $\sigma=80$. For the first-order method, we use $\Delta t=10^{-3}$, chosen near the upper end of the stable range. For the inexact experiments, we isolate the two sources of approximation error. We first set $\delta\approx0$ and vary $\eta\in\{0.1,0.05,0.02,0.01\}$. We then set $\eta\approx 0$ and vary $\delta\in\{0.02,0.01,0.005,0.001\}$. 

In \Cref{fig:2}(a), we compare the performance of the second-order method in \Cref{algorithm 1.1} with that of the first-order method in \eqref{eq: first-order method}. The second-order method reaches the prescribed accuracy within only tens of iterations and requires less than $0.1$ seconds of CPU time. In contrast, the first-order method requires thousands of iterations and several seconds of CPU time to achieve the same accuracy, demonstrating the substantial computational advantage of the second-order method. In \Cref{fig:2}(b), we plot $\|\nabla f(\x_{j+1})\|_2$ against $\|\nabla f(\x_j)\|_2$. The tail of the curve exhibits a slope of approximately $2$ on the log-log scale, confirming the quadratic convergence rate predicted by the theory. In \Cref{fig:2}(c-d), we consider \Cref{algorithm 1.1} with inexact unstable subspaces and inexact solutions of the subproblem, where $\eta$ and $\delta$ characterize the corresponding approximation accuracies. In this setting, the asymptotic convergence rate becomes linear. Nevertheless, the linear convergence remains very fast because the corresponding contraction factor depends on the parameters $\delta$ and $\eta$ in \Cref{thm:inexact-recursion}, both of which are small in practice. Indeed, the algorithm still reaches the prescribed accuracy within tens of iterations in \Cref{fig:2}(c-d). Moreover, as the approximations of the unstable subspace and the subproblem solution become more accurate (or $\delta$ and $\eta$ decrease), the convergence becomes progressively faster, consistent with the dependence characterized in \Cref{thm:inexact-recursion}.

\subsection{Landau-de Gennes energy functional}\label{subsec: LdG}

\begin{figure}
    \centering
    \includegraphics[width=0.99\linewidth]{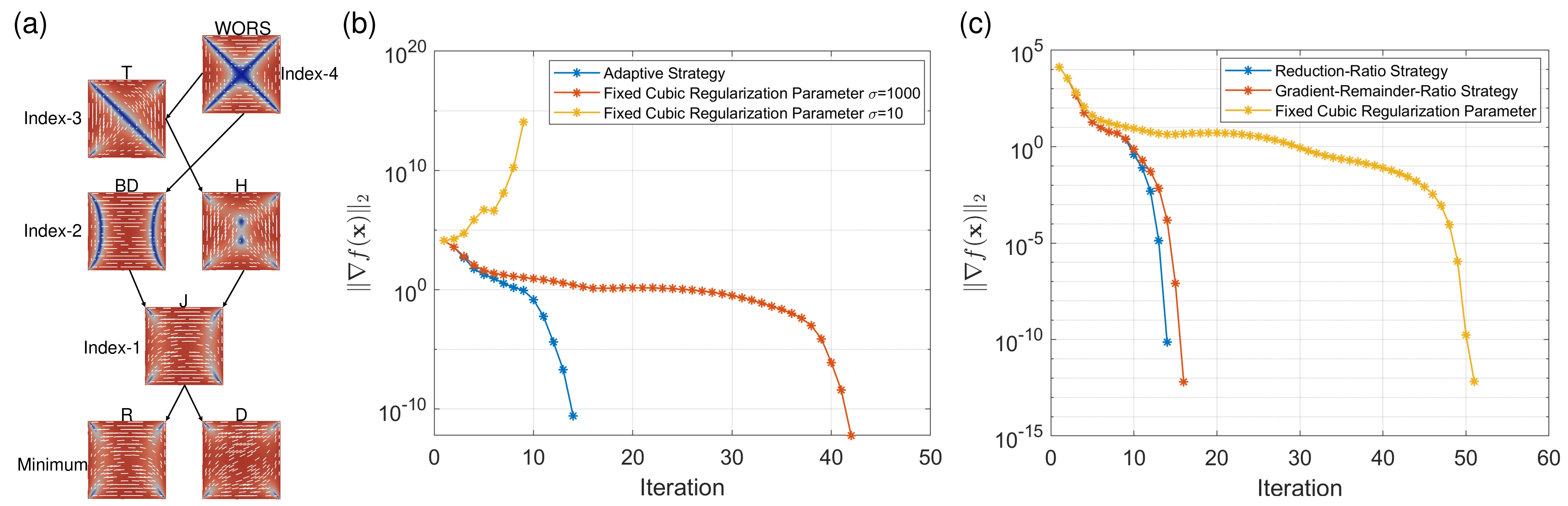}
    \caption{
Solution landscape and performance of the adaptive cubic regularization strategy for the discrete Landau--de Gennes energy functional.
(a) Solution landscape for $\lambda^2=30$, up to symmetry. The color bar labels the order parameter $\sqrt{|\mathbf{Q}|^2/2}$ and the white lines represent the directors, i.e., the eigenvector corresponding to the largest eigenvalue of $\mathbf{Q}$.
(b) Convergence histories for computing an index-$2$ saddle point from a random initial condition using the adaptive strategy and fixed cubic regularization parameters $\sigma=1000$ and $\sigma=10$. (c) Convergence histories for computing a local minimum from a random initial condition using the reduction-ratio strategy, our gradient-remainder-ratio strategy in \eqref{eq: adaptive parameter}, and a fixed cubic regularization parameter.}
    \label{fig:3}
\end{figure}

The Landau--de Gennes (LdG) energy functional is a continuum model for nematic liquid crystals (NLCs), in which the orientational order of the NLC state is described by a macroscopic $\mathbf{Q}$-tensor order parameter. The non-dimensionalized reduced LdG energy functional is given by \cite{wang2019order}
$$
  E[\mathbf{Q}(x,y)]: = \int_{\mathcal{D}} \dfrac{1}{2}|\nabla \mathbf{Q}(x,y)|^2+ \lambda^2\left(-\dfrac{B^2}{8C^2}\mathrm{tr}(\mathbf{Q}(x,y)^2)+\dfrac{1}{8}\left(\mathrm{tr}(\mathbf{Q}(x,y)^2)\right)^2\right) \mathrm{d}x \mathrm{d}y.
$$
where $\mathcal{D}=[-1,1]^2$ is the non-dimensionalized domain. We set  $\lambda^2=30$, while the other parameters are kept the same as in \cite{yin2020construction}. We impose tangential Dirichlet boundary conditions, requiring the liquid crystal molecules to align along the boundary, consistent with the experimental setting described in \cite{tsakonas2007multistable}.
We use a finite-difference method for the spatial discretization with mesh size $\delta x=1/32$. The tensor field $\mathbf{Q}(x,y)$ is then represented by a $d$-dimensional vector $\x$, where $d$ is determined by the number of spatial grid points, and hence depends on $\delta x$; 
the energy functional $E[\mathbf{Q}(x,y)]$ is discretized into a function $f(\x)$. The discrete inner product is defined by $\langle \x,\y\rangle_{\delta x} = \delta x^2 \x^\top \y$, with the corresponding norm denoted by $\|\cdot\|_{\delta x}$. The invariant subspace associated with the $k$ smallest eigenvalues is computed
iteratively and warm-started using the LOBPCG. The cubic subproblem is solved in a Krylov subspace of dimension $200$. For the adaptive cubic regularization strategy, we initialize
$\sigma_0=1000$ and use $
\gamma_{\rm dec}=0.6,
\gamma_{\rm inc}=2, d_l=0.5,
d_u=1, \sigma_{\min}=0.1, \sigma_{\max}=10^4$.

Starting from a random initial condition, \Cref{algorithm 1.1} with $k=4$ locates the index-$4$ saddle point, represented by WORS, which is characterized by two line defects along the diagonals of the square. Starting from WORS, we then search for lower-index critical points in descending order of Morse index to construct the solution landscape, as described in \Cref{sec: solution landscape}. Up to symmetry, we identify seven distinct critical points for $k\leq 4$, as shown in \Cref{fig:3}(a). These include two local minima, referred to as the diagonal (D) state and the rotational (R) state, which correspond to experimentally observed stable configurations in square confinements \cite{tsakonas2007multistable}.

In \Cref{fig:3}(b-c), we examine the performance of our adaptive strategy for selecting the cubic regularization parameter. In \Cref{fig:3}(b), we consider the computation of an index-$2$ saddle point from a random initial condition. For a fixed cubic regularization parameter, if $\sigma$ is relatively small, e.g., $\sigma=10$, the algorithm becomes unstable and the iterates diverge, as shown by the yellow curve. This occurs because the step size is excessively large during the initial iterations, so that the local model in the subproblem no longer provides an adequate approximation to $f$. In contrast, when $\sigma$ is very large, e.g., $\sigma=1000$, the algorithm remains stable but converges slowly, as illustrated by the orange curve. The blue curve corresponds to the adaptive strategy in \eqref{eq: adaptive parameter}, which achieves a better balance between robustness and efficiency and therefore reaches the saddle point in fewer iterations than the method with a large fixed value of $\sigma$. In \Cref{fig:3}(c),
for comparison purposes,
we consider the computation of an index-$0$ critical point, i.e., a local minimum,  starting from a random initial condition. This setting allows us to compare our adaptive strategy with the standard adaptive strategy commonly used in cubic-regularization methods for optimization. Both adaptive strategies outperform the methods with fixed regularization parameters and exhibit comparable performance. This suggests that our approach may also serve as an alternative adaptive strategy for optimization, while having the advantage that it extends naturally to saddle-point computations.

\section{Conclusion and discussion}\label{sec: conclusion}
In this paper, we propose a second-order algorithm for locating critical points of a prescribed index and for constructing solution landscapes. Under suitable assumptions, we establish a quadratic local convergence rate when the cubic subproblem and the unstable subspace are computed exactly. For inexact computations, we derive a linear convergence result whose contraction factor depends explicitly on the approximation accuracies of the unstable subspace and the subproblem solution. We further show that the fixed points of the proposed algorithm are precisely index-$k$ critical points, which distinguishes our method from the classical Newton method and provides a direct mechanism for targeting critical points of a prescribed Morse index. A key algorithmic advance of the proposed method is that its subproblem has the same form as that arising in cubic-regularization methods for optimization. This allows existing efficient solvers for cubic subproblems to be directly incorporated into our framework, making the method suitable for large-scale computations and applications to high-dimensional systems. We also introduce an adaptive strategy for updating the cubic regularization parameter, which improves the robustness and efficiency of the algorithm. The numerical experiments confirm the theoretical convergence results, demonstrate the effectiveness of the adaptive strategy, and illustrate the practical capability of the method for constructing solution landscapes.

Several interesting directions remain for future investigation. The present framework may be extended directly to trust-region methods, since cubic regularization and trust-region methods share closely related mechanisms for controlling the step size and ensuring the reliability of local quadratic models. Furthermore, variants of cubic-regularization and trust-region methods have been developed for large-scale optimization, such as stochastic \cite{tripuraneni2018stochastic} and subsample \cite{kohler2017sub} models. Adapting these ideas to the computation of critical points with prescribed index may further broaden the applicability and scalability of the proposed framework.

\normalem
\bibliographystyle{siamplain}
\bibliography{references}

\end{document}